\documentclass[final,3p,times,10pt]{elsarticle}
\usepackage[centertags]{amsmath}
\usepackage{amsmath,amsfonts,amssymb,amsthm,bbold,epstopdf,newlfont,graphicx}
\usepackage{booktabs,bm,enumitem,setspace,empheq,cancel}
\usepackage{algorithm}
\usepackage[short,nocomma]{optidef}
\usepackage{algpseudocode}
\usepackage{caption,subcaption}
\usepackage{threeparttable,multirow}
\usepackage{amsmath,accents}
\usepackage[T1]{fontenc}
\usepackage[american]{babel}
\usepackage{natbib}
\biboptions{comma,square,sort&compress}
\usepackage{appendix}
\usepackage[colorlinks=true]{hyperref}

\newcommand\hcancel[2][black]{\setbox0=\hbox{$#2$}%
\rlap{\raisebox{.25\ht0}{\textcolor{#1}{\rule{0.7\wd0}{0.75pt}}}}#2} 
\newcommand\hcancelt[2][black]{\setbox0=\hbox{$#2$}%
\rlap{\raisebox{.25\ht0}{\textcolor{#1}{\hspace{0.3mm}\rule{0.7\wd0}{0.75pt}}}}#2} 
\newtheorem{thm}{Theorem}[section]

\newtheorem{cor}{Corollary}[section]
\newtheorem{rem}{Remark}[section]
\theoremstyle{definition}

\numberwithin{algorithm}{section}
\numberwithin{equation}{section}
\renewcommand{\theequation}{\thesection.\arabic{equation}}
\def\simgt{\,\hbox{\lower0.6ex\hbox{$>$}\llap{\raise0.3ex\hbox{$\sim$}}}\,}
\def\simlt{\,\hbox{\lower0.6ex\hbox{$<$}\llap{\raise0.3ex\hbox{$\sim$}}}\,}
\def\simgteq{\,\hbox{\lower0.6ex\hbox{$\ge$}\llap{\raise0.6ex\hbox{$\sim$}}}\,}
\def\simlteq{\,\hbox{\lower0.6ex\hbox{$\le$}\llap{\raise0.6ex\hbox{$\sim$}}}\,}
\def\applteq{\,\hbox{\lower0.6ex\hbox{$\le$}\llap{\raise0.8ex\hbox{$\approx$}}}\,}
\def\applt{\,\hbox{\lower0.6ex\hbox{$<$}\llap{\raise0.5ex\hbox{$\approx$}}}\,}
\DeclareMathAlphabet\mathbfcal{OMS}{cmsy}{b}{n}

\DeclareMathOperator*{\argmax}{argmax}

\DeclareMathOperator{\erf}{erf}
\DeclareMathOperator{\erfi}{erfi}

\usepackage{cuted}
\makeatletter
\def\user@resume{resume}
\def\user@intermezzo{intermezzo}
\newcounter{previousequation}
\newcounter{lastsubequation}
\newcounter{savedparentequation}
\renewenvironment{subequations}[1][]{%
      \def\user@decides{#1}%
      \setcounter{previousequation}{\value{equation}}%
      \ifx\user@decides\user@resume 
           \setcounter{equation}{\value{savedparentequation}}%
      \else  
      \ifx\user@decides\user@intermezzo
           \refstepcounter{equation}%
      \else
           \setcounter{lastsubequation}{0}%
           \refstepcounter{equation}%
      \fi\fi
      \protected@edef\theHparentequation{%
          \@ifundefined {theHequation}\theequation \theHequation}%
      \protected@edef\theparentequation{\theequation}%
      \setcounter{parentequation}{\value{equation}}%
      \ifx\user@decides\user@resume 
           \setcounter{equation}{\value{lastsubequation}}%
         \else
           \setcounter{equation}{0}%
      \fi
      \def\theequation  {\theparentequation  \alph{equation}}%
      \def\theHequation {\theHparentequation \alph{equation}}%
      \ignorespaces
}{%
  \ifx\user@decides\user@resume
       \setcounter{lastsubequation}{\value{equation}}%
       \setcounter{equation}{\value{previousequation}}%
  \else
  \ifx\user@decides\user@intermezzo
       \setcounter{equation}{\value{parentequation}}%
  \else
       \setcounter{lastsubequation}{\value{equation}}%
       \setcounter{savedparentequation}{\value{parentequation}}%
       \setcounter{equation}{\value{parentequation}}%
  \fi\fi
  \ignorespacesafterend
}
\makeatother
\newcommand{\C}[1]{\mathcal{#1}}

\newcommand{\F}[1]{\mathbf{#1}}
\newcommand{\bs}[1]{\boldsymbol{#1}}
\newcommand{\bsC}[1]{\boldsymbol{\C{#1}}}

\newcommand{\MB}[1]{\mathbb{#1}}
\newcommand{\MBS}{\MB{S}}
\newcommand{\MBG}{\MB{G}}
\newcommand{\MBSG}{\hat{\MBG}}
\newcommand{\MBR}{\mathbb{R}}

\newcommand{\MBRP}{\MBR^+}
\newcommand{\MBRzer}{\MBR_0}
\newcommand{\MBRzerP}{\MBRzer^+}
\newcommand{\MBZ}{\mathbb{Z}}
\newcommand{\MBZP}{\MBZ^+}
\newcommand{\MBZzer}{\MBZ_0}
\newcommand{\MBZzerP}{\MBZzer^+}
\newcommand{\MBZe}{\MBZ_e}
\newcommand{\MBZeP}{\MBZe^+}
\newcommand{\MBZzereP}{\MBZ_{0,e}^+}
\newcommand{\MBZOP}{\MBZ_{o}^+}
\newcommand{\MBT}{\mathbb{T}}
\newcommand{\MBJ}{\mathbb{J}}
\newcommand{\MBN}{\mathbb{N}}

\newcommand{\MFC}{\mathfrak{C}}
\newcommand{\MFK}{\mathfrak{K}}

\newcommand{\cancbra}[1]{\hcancel{[}#1\hcancelt{]}}
\newcommand{\sumd}{\sideset{}{'}}

\newcommand{\bmx}{\bm{x}}

\newcommand{\bmt}{\bm{t}}

\newcommand{\bmy}{\bm{y}}

\newcommand{\bmh}{\bm{h}}

\newcommand{\bmzer}{\bm{\mathit{0}}}
\newcommand{\bmone}{\bm{\mathit{1}}}
\newcommand{\bmb}{\bm{b}}

\newcommand{\hz}{\hat{z}}

\newcommand{\MBOmega}{\MB{\Omega}}

\newcommand{\foralla}{\,\forall_{\mkern-6mu a}\,}
\newcommand{\foralle}{\,\forall_{\mkern-6mu e}\,}
\newcommand{\foralls}{\,\forall_{\mkern-6mu s}\,}

\newcommand{\Def}[1]{\text{Def}\left(#1\right)}

\newcommand{\GLD}[4]{{}_{\;\;#1}^{GL}D_{#2}^{#3}{#4}}
\newcommand{\RLD}[4]{{}_{\;\;#1}^{RL}D_{#2}^{#3}{#4}}
\newcommand{\CD}[4]{{}_{#1}^CD_{#2}^{#3}{#4}}
\newcommand{\MRLD}[4]{{}_{\hspace{3.6mm}#1}^{MRL}D_{#2}^{#3}{#4}}
\newcommand{\MCD}[4]{{}_{\;\;#1}^{MC}D_{#2}^{#3}{#4}}
\newcommand{\MD}[4]{{}_{#1}^{M}D_{#2}^{#3}{#4}}
\newcommand{\ED}[4]{{}_{#1}^{E}D_{#2}^{#3}{#4}}
\usepackage{mathtools}
\mathtoolsset{showonlyrefs}
\makeatletter
\def\BState{\State\hskip-\ALG@thistlm}
\makeatother
\MakeRobust{\eqref}
\errorcontextlines\maxdimen

\makeatletter
    \newcommand*{\algrule}[1][\algorithmicindent]{\makebox[#1][l]{\hspace*{.5em}\thealgruleextra\vrule height \thealgruleheight depth \thealgruledepth}}%
\newcommand*{\thealgruleextra}{}
\newcommand*{\thealgruleheight}{.75\baselineskip}
\newcommand*{\thealgruledepth}{.25\baselineskip}

\newcount\ALG@printindent@tempcnta
\def\ALG@printindent{%
    \ifnum \theALG@nested>0
        \ifx\ALG@text\ALG@x@notext
        \else
            \unskip
            \addvspace{-1pt}
            \ALG@printindent@tempcnta=1
            \loop
                \algrule[\csname ALG@ind@\the\ALG@printindent@tempcnta\endcsname]%
                \advance \ALG@printindent@tempcnta 1
            \ifnum \ALG@printindent@tempcnta<\numexpr\theALG@nested+1\relax
            \repeat
        \fi
    \fi
    }%
\usepackage{etoolbox}
\patchcmd{\ALG@doentity}{\noindent\hskip\ALG@tlm}{\ALG@printindent}{}{\errmessage{failed to patch}}
\makeatother

\newbox\statebox
\newcommand{\myState}[1]{%
    \setbox\statebox=\vbox{#1}%
    \edef\thealgruleheight{\dimexpr \the\ht\statebox+1pt\relax}%
    \edef\thealgruledepth{\dimexpr \the\dp\statebox+1pt\relax}%
    \ifdim\thealgruleheight<.75\baselineskip
        \def\thealgruleheight{\dimexpr .75\baselineskip+1pt\relax}%
    \fi
    \ifdim\thealgruledepth<.25\baselineskip
        \def\thealgruledepth{\dimexpr .25\baselineskip+1pt\relax}%
    \fi
    \State #1%
    \def\thealgruleheight{\dimexpr .75\baselineskip+1pt\relax}%
    \def\thealgruledepth{\dimexpr .25\baselineskip+1pt\relax}%
}
\makeatletter
\newcommand{\oset}[3][0ex]{%
  \mathrel{\mathop{#3}\limits^{
    \vbox to#1{\kern-2\ex@
    \hbox{$\scriptstyle#2$}\vss}}}}
\makeatother
\begin{document}
\begin{frontmatter}
\title{Fourier-Gegenbauer Pseudospectral Method for Solving Time-Dependent One-Dimensional Fractional Partial Differential Equations with Variable Coefficients and Periodic Solutions}
\author[XMUM,Assiut]{Kareem T. Elgindy}
\ead{kareem.elgindy@xmu.edu.my}
\address[XMUM]{Mathematics Department, School of Mathematics and Physics, Xiamen University Malaysia, Sepang 43900, Malaysia}
\address[Assiut]{Mathematics Department, Faculty of Science, Assiut University, Assiut 71516, Egypt}
\begin{abstract}
In this paper, we present a novel pseudospectral (PS) method for solving a new class of initial-value problems (IVPs) of time-dependent one-dimensional fractional partial differential equations (FPDEs) with variable coefficients and periodic solutions. A main ingredient of our work is the use of the recently developed periodic RL/Caputo fractional derivative (FD) operators with sliding positive fixed memory length of \citet{bourafa2021periodic} or their reduced forms obtained by \citet{elgindy2023fourier} as the natural FD operators to accurately model FPDEs with periodic solutions. The proposed method converts the IVP into a well-conditioned linear system of equations using the PS method based on Fourier collocations and Gegenbauer quadratures. The reduced linear system has a simple special structure and can be solved accurately and rapidly by using standard linear system solvers. A rigorous study of the computational storage requirements as well as the error and convergence of the proposed method is presented. The idea and results presented in this paper are expected to be useful in the future to address more general problems involving FPDEs with periodic solutions.
\end{abstract}
\end{frontmatter}

\section{Introduction}
\label{Int}
Fractional partial differential equations (FPDEs) have become the natural mathematical models to model various problems and applications compared with integer-order partial differential equations (PDEs) due to their high accuracy in addition to the flexibility and non-locality of fractional derivatives (FDs) compared with classical integer-order derivatives. In fact, FDs and FPDEs have played a very significant roles in chemistry, chemical engineering, geomechanics, computer vision, Coronavirus disease spread, hydrological processes, oceanography, decision and control, nuclear energy, medicine and surgery; cf. \cite{khan2023level,sioofy2022proposed,lisha2023analytical,su2023random,
arefin2022investigation,asjad2023applications,owolabi2022fractal,
mascarenhas2022stochastic,hamada2022nonlinear,ibrahim2022medical}. 

In contrast to integer-order derivatives, classical Riemann-Liouville (RL) and Caputo FDs of a non-constant $T$-periodic function are not $T$-periodic functions, limiting their use to model real periodic phenomena and prompting the need for research advancements in this substantial field of science. In this study, we provide the first successful attempt in the literature to model FPDEs with periodic solutions using periodic FD operators that can preserve the periodicity of a periodic function and allow for the existence of periodic solutions to FPDEs. In particular, we model FPDEs with periodic solutions using the recently developed periodic FD operator of \citet{elgindy2023fourier}, which is a useful reduced form of an earlier periodic FD inaugurated by \citet{bourafa2021periodic}. The employed FD operator is a useful modification of the classical RL and Caputo FD operators by fixing their memory length and varying their lower terminals. The reduced FD operator developed in \cite{elgindy2023fourier} allows accurate computation of the singular integral of the FD formula defined in \cite{bourafa2021periodic} by removing the singularity prior to numerical integration using a smart change of variables and renders the reduced integral well behaved. In fact, the introduced transformation largely simplifies the problem of calculating the periodic FDs of periodic functions to the problem of evaluating the integral of the first derivatives of their trigonometric Lagrange interpolating polynomials, which can be treated accurately and efficiently using Gegenbauer
quadratures. This approach, together with Fourier collocations and interpolations celebrated for their stability, rapid convergence, and cost efficiency, constitutes the core of our proposed Fourier-Gegenbauer (FG) based pseudospectral (PS) method (FGPS method): A novel PS method based on Fourier collocation and Gegenbauer quadratures for solving FPDEs with periodic solutions. In particular, the proposed method converts the initial-value problems (IVPs) of FPDEs with periodic solutions into well-conditioned linear systems of equations using FG-based PS technology. The reduced linear systems have sparse block global coefficient matrices that can be generated efficiently using ``smart'' index matrix mappings, which allow us to solve the reduced linear systems very accurately and rapidly using standard linear system solvers. Although the use of Fourier and Jacobi polynomials for solving FPDEs is not novel and appeared in a number of works such as \cite{li2015numerical,li2019theory,doha2011chebyshev,bueno2014fourier}, to the best of our knowledge, the current study not only provides the first successful attempt to model FPDEs with periodic solutions using periodic FD operators but also provides the first efficient, stable, and highly accurate numerical method for solving them. We confine our work to FDs with a fractional-order $0 < \alpha < 1$; however, the current work can be easily extended to cover higher-order FPDEs with periodic solutions. The proposed FGPS method is an extension to classical PS methods, which are considered to be one of the biggest technologies for solving PDEs that were
largely developed about a half century ago \cite{gottlieb1977numerical,fornberg1994review,Fornberg1996practical,
boyd2001chebyshev,ElgindyHareth2023a}. For a clear exposition of PS methods exhibiting a wide range of outlooks
on the subject, the reader may consult, to mention a few, the books \cite{Fornberg1996practical,kopriva2009}. A comprehensive review on recent progress on Fourier PS methods can be found in \cite{elgindy2019high,Elgindy2023a,Elgindy2023b,
elgindy2023optimal,elgindy2023fourier}. For a survey on Gegenbauer polynomials and quadratures and their relevant theory, the reader may consult \cite{Elgindy2013b,elgindy2018highb,elgindy2018optimal,Elgindy2019b,
ElgindyHareth2023a}, and the references therein.

The remainder of this paper is organized as follows. In Section \ref{sec:PN}, we provide some primary definitions of FDs and notations to be used in the paper’s presentation. In Section \ref{sec:PS1}, the IVP of the FPDE under study is introduced in general form. In Section \ref{sec:FGPM1}, we present the FGPS method for solving the IVP. A study on the computational storage requirements of the FGPS method is presented in Section \ref{sec:CS1}. In Section \ref{sec:ECA1}, we present error and convergence analyses of the FGPS method. The performance of the FGPS method is demonstrated in Section \ref{sec:NS1}. Finally, Section \ref{sec:Conc} provides concluding remarks.


\section{Preliminaries and Notations}
\label{sec:PN}
The following notations are used throughout this study to abridge and simplify the mathematical formulas. Many of these
notations appeared earlier in \cite{Elgindy2023a,elgindy2023fourier}; however, for convenience and to keep the paper self-explanatory, we summarize them below together with the new notations.

\noindent\textbf{Logical Symbols.} The  symbols $\forall, \foralla, \foralle$, and $\foralls$ stand for the phrases ``for all,'' ``for any,'' ``for each,'' and ``for some,'' in respective order.\\[0.5em]
\textbf{List and Set Notations.} $\MFC$ denotes the set of all complex-valued functions. Moreover, $\MBR$, $\MBZ, \MBZP, \MBZzerP, \MBZOP, \MBZeP$, and $\MBZzereP$ denote the sets of real numbers, integers, positive integers, non-negative integers, positive odd integers, positive even integers, and non-negative even integers, respectively. When we overset any of the above sets by a right arrow we mean the subset of that set containing sufficiently large numbers; for example, $\oset{\rightarrow}{\MBZP}$ stands for the set of all sufficiently large positive integers. The notations $i$:$j$:$k$ or $i(j)k$ indicate a list of numbers from $i$ to $k$ with increment $j$ between numbers, unless the increment equals one where we use the simplified notation $i$:$k$. For example, $0$:$0.5$:$2$ simply means the list of numbers $0, 0.5, 1, 1.5$, and $2$, while $0$:$2$ means $0, 1$, and $2$. The list of symbols $y_1, y_2, \ldots, y_n$ is denoted by $\left. y_i \right|_{i=1:n}$ or simply $y_{1:n}$, and their set is represented by $\{y_{1:n}\}\,\foralla n \in \MBZP$; the same set excluding $y_j$ is denoted by $\{y_{1:n}\}_{\neq y_j}\,\foralla j \in \MBN_n$. The list of ordered pairs $(x_{m,0},t_{n,0}), (x_{m,0},t_{n,1})$, $\ldots, (x_{m,1},t_{n,0}), \ldots, (x_{m,m-1},t_{n,n-1})$ is denoted by $(x_{m,0:m-1}, t_{n,0:n-1})$ and their set is denoted by $\{(x_{m,0:m-1}, t_{n,0:n-1})\}$. We define $\MBJ_n = \{0:n-1\}$ and $\MBN_n = \{1:n\}\,\foralla n \in \MBZP$. $\MBS_n^{T} = \left\{t_{n,0:n-1}\right\}$ is the set of $n$ equally-spaced points such that $t_{n,j} = T j/n\, \forall j \in \MBJ_n$. For a set of ordered pairs, we define $\MBS_{m,n}^{T_1,T_2} = \{(x_{m,0:m-1}, t_{n,0:n-1}): x_{m,l} \in \MBS_{m}^{T_1}, t_{n,j} \in \MBS_{n}^{T_2}\,\foralla (l,j) \in \MBJ_m \times \MBJ_n, \{m,n\} \subset \MBZP, \{T_1,T_2\} \subset \MBRP$\}. $\MBG_n^{\lambda} = \left\{z_{n,0:n}^{\lambda}\right\}$ is the set of Gegenbauer-Gauss (GG) zeros of the $(n+1)$st-degree Gegenbauer polynomial with index $\lambda > -1/2$, and $\MBSG_{1,n}^{\lambda} = \left\{\hz_{n,0:n}^{\lambda}: \hz_{n,0:n}^{\lambda} = \frac{1}{2} \left(z_{n,0:n}^{\lambda}+1\right)\right\}$ is the shifted Gegenbauer-Gauss (SGG) points set in the interval $[0,1]\,\foralla n \in \MBZP$; cf. \cite{Elgindy201382,elgindy2018high,elgindy2018optimal}. The specific interval $[0, T]$ is denoted by $\MBOmega_{T}\,\forall T > 0$; for example, $[0, t_{n,j}]$ is denoted by ${\MBOmega_{t_{n,j}}}\,\forall j \in \MBJ_n$. 
Cartesian products like $\MBOmega_{T_1} \times \MBOmega_{T_2}, \MBJ_{n_1} \times \MBJ_{n_2}$, and $\MBN_{n_1} \times \MBN_{n_2}$ are denoted by $\MBOmega_{T_1, T_2}, \MBJ_{n_1,n_2}$, and $\MBN_{n_1,n_2}$, respectively, $\foralla \{T_1,T_2\} \subset \MBRP$ and $\{n_1,n_2\} \subset \MBZP$.\\[0.5em] 
\textbf{Function Notations.} $\left\lceil  \cdot  \right\rceil, \left\lfloor {\cdot} \right\rfloor$ and $\Gamma$ denote the ceil, floor, and Gamma functions, respectively. $\left( {\begin{array}{*{20}{c}}
\alpha \\
k
\end{array}} \right)$ is the binomial coefficient indexed by the pair $\alpha \in \MBR$ and $k \in \MBZzerP$. $(\alpha)^{(k)} = \prod\limits_{l = 0}^{k-1} {(\alpha  - l)}$ is the $k$th-factorial power (the falling factorial) function of $\alpha \in \MBR$. For convenience, we shall denote $g(t_{n})$ by $g_n \foralla g \in \MFC, n \in \MBZ, t_n \in \MBR$, unless stated otherwise. We extend this writing convention to multidimensional functions; for example, to evaluate a bivariate function $u$ at some discrete points set $\left\{(x_{0:m-1},t_{0:n-1})\right\}$, we write $u_{m,n}$ to simply mean $u(x_m,t_n)\,\forall x_m \in \MBS_m^{T_1}, t_n \in \MBS_n^{T_2}$. $u_{0:m-1,0:n-1}$ stands for the list of function values $u_{0,0}, u_{0,1}, \ldots, u_{0,n-1}, u_{1,0}, \ldots, u_{1,n-1}, \ldots, u_{m-1,n-1}$. Finally, by $u$ is a ${}_xT_1$-${}_tT_2$-periodic function we mean $u$ is a $T_1$-periodic and a $T_2$-periodic function with respect to $x$ and $t$, respectively, $\foralla u: \MBR^2 \to \MBR$.\\[0.5em]
\textbf{Integral Notations.} We denote $\int_0^{b} {h(t)\,dt}$ and $\int_a^{b} {h(t)\,dt}$ by $\C{I}_{b}^{(t)}h$ and $\C{I}_{a, b}^{(t)}h$, respectively, $\foralla$ integrable $h \in \MFC, \{a, b\} \subset \MBR$. If the integrand function $h$ is to be evaluated at any other expression of $t$, say $u(t)$, we express $\int_0^{b} {h(u(t))\,dt}$ and $\int_a^b {h(u(t))\,dt}$ with a stroke through the square brackets as $\C{I}_{b}^{(t)}h\cancbra{u(t)}$ and $\C{I}_{a,b}^{(t)}h\cancbra{u(t)}$ in respective order.\\[0.5em] 
\textbf{Space and Norm Notations.} $\MBT_{T}$ is the space of $T$-periodic, univariate functions $\foralla T \in \MBRP$. $\Def{\MBOmega}$ is the space of functions defined on $\MBOmega$. $C^k(\MBOmega)$ is the space of $k$ times continuously differentiable functions on ${\MBOmega}\,\forall k \in \MBZzerP$ with the common understanding that $C(\MBOmega)$ means $C^0(\MBOmega)$. We define the following two spaces:
\begin{gather*}
{}^{1}_{T}\mathfrak{X}_{n_1}^{n_2} = \{[y_0, \ldots, y_{n_1}]^{\top}: \MBR \to \MBR^{n_1}\text{ s.t. } y_j \in \MBT_{T} \cap C^{n_2}(\MBR)\,\forall j \in \MBN_{n_1}\},\\
{}^{\quad 2}_{T_1,T_2}\mathfrak{X}_{n_1}^{n_2} = \{[y_0, \ldots, y_{n_1}]^{\top}: \MBR^2 \to \MBR^{n_1}\text{ s.t. } y_j(x,t) = z \in \MBR, y_j \in \MBT_{T_1} \cap C^{n_2}(\MBR) \foralla t \in \MBOmega_{T_2}\text{ and }y_j \in \MBT_{T_2} \cap C^{n_2}(\MBR) \foralla x \in \MBOmega_{T_1}\,\forall j \in \MBN_{n_1}\},
\end{gather*}
where ${}^{1}_{T}\mathfrak{X}_{n_1}^{n_2}$ is the space of $T$-periodic, $n_2$-times continuously differentiable, $n_1$-dimensional single-variable vector functions on $\MBR$, and ${}^{\quad 2}_{T_1,T_2}\mathfrak{X}_{n_1}^{n_2}$ is the space of ${}_xT_1$-${}_tT_2$-periodic, $n_2$-times continuously differentiable, $n_1$-dimensional bivariate vector functions on $\MBR^2$. $L^p({\MBOmega})$ is the Banach space of measurable functions $u$ defined on ${\MBOmega}$ such that ${\left\| u \right\|_{{L^p}}} = {\left( {{\C{I}_{\MBOmega}}{{\left| u \right|}^p}} \right)^{1/p}} < \infty\,\forall p \ge 1$. In particular, we write $\left\|u\right\|_{\infty}$ to denote ${\left\| u \right\|_{{L^{\infty}}}}$. Finally, for vector arguments, $\left\|\cdot\right\|_2$ and $\left\|\cdot\right\|_{\infty}$ denote the usual Euclidean and infinity norms of vectors.\\[0.5em]
\textbf{Vector Notations.} We shall use the shorthand notations $\bmt_N$ and $g_{0:N-1}$ to stand for the column vectors $[t_{N,0}, t_{N,1}, \ldots$, $t_{N,N-1}]^{\top}$ and $[g_0, g_1, \ldots, g_{N-1}]^{\top}\,\forall N \in \MBZP$ in respective order. In general, $\foralla h \in \MFC$ and vector $\bmy$ whose $i$th-element is $y_i \in \MBR$, and is denoted by $\bmy(i)$, the notation $h(\bmy)$ stands for a vector of the same size and structure of $\bmy$ such that $h(y_i)$ is the $i$th element of $h(\bmy)$. Moreover, by $\bmh(\bmy)$ or $h_{1:m}\cancbra{\bmy}$ with a stroke through the square brackets, we mean $[h_1(\bmy), \ldots, h_m(\bmy)]^{\top}\,\foralla m$-dimensional column vector function $\bmh$, with the realization that the definition of each array $h_i(\bmy)$ follows the former notation rule $\foralle i$. Furthermore, if $\bmy$ is a vector function, say $\bmy = \bmy(t)$, then we write $h(\bmy(\bmt_N))$ to denote $[h(\bmy(t_{N,0})), h(\bmy(t_{N,1})), \ldots$, $h(\bmy(t_{N,N-1}))]^{\top}$.\\[0.5em] 
\textbf{Matrix Notations.} $\F{O}_n, \F{1}_n$, and $\F{I}_n$ stand for the zero, all ones, and the identity matrices of size $n$. $\F{C}_{n,m}$ indicates that $\F{C}$ is a rectangular matrix of size $n \times m$; moreover, $\F{C}_n$ denotes a row vector whose elements are the $n$th-row elements of $\F{C}$, except when $\F{C}_n = \F{O}_n, \F{1}_n$, or $\F{I}_n$, where it denotes the size of the matrix. For a two-dimensional matrix $\F{C}$, the notation $\F{C}_{\hcancel{i}}$ stands for the matrix obtained by deleting the $i$th-row of $\F{C}$. Also, $\F{C}_{i,\hcancel{j}}$ is the row vector obtained by deleting the $j$th-column of $\F{C}_i$. The common notation $\F{C}(i,j)$ refers to the $(i, j)$ entry of $\F{C}$. For convenience, a vector is represented in print by a bold italicized symbol while a two-dimensional matrix is represented by a bold symbol, except for a row vector whose elements form a certain row of a matrix where we represent it in bold symbol as stated earlier. For example, $\bmone_n$ and $\bmzer_n$ denote the $n$-dimensional all ones- and zeros- column vectors, while $\F{1}_n$ and $\F{O}_n$ denote the all ones- and zeros- matrices of size $n$, respectively. The notation $[.;.]$ denotes the usual vertical concatenation. Finally, $\kappa(\F{A})$ denotes the condition number of a matrix $\F{A}$.\\[0.5em]
\textbf{Common Fractional Differentiation Formulas.} Let $\alpha \in \MBRP$, $m = \left\lceil  \alpha  \right\rceil$, and $f \in \Def{\MBOmega_T}$. The $\alpha$-th order Gr\"{u}nwald-Letnikov derivative of $f$ with respect to $t$ and a terminal value $a$ is given by
\begin{equation}
\GLD{a}{t}{\alpha}{f(t)} = \mathop {\lim }\limits_{\scriptstyle h \to 0\atop
\scriptstyle nh = t - a} {h^{ - \alpha }}\sum\limits_{k = 0}^n {{{( - 1)}^k}\left( {\begin{array}{*{20}{c}}
\alpha \\
k
\end{array}} \right)f(t - kh)}.
\end{equation}
The $\alpha$-th order left RL and Caputo FDs are denoted by $\RLD{0}{t}{\alpha}{f(t)}$ and $\CD{0}{t}{\alpha}{f(t)}$, respectively, and are defined for $t \in \MBOmega_T$ by
\begin{align}
\RLD{0}{t}{\alpha}{f(t)} &= \left\{ \begin{array}{l}
\frac{1}{{\Gamma (m - \alpha )}}\frac{{{d^m}}}{{d{t^m}}}\C{I}_{t}^{(\tau)}\left[{{{(t - \tau )}^{m - \alpha  - 1}}f}\right],\quad \alpha \notin \MBZP,\\
f^{(m)}(t),\quad \alpha \in \MBZP,
\end{array} \right.\\
\CD{0}{t}{\alpha}{f(t)} &= \left\{ \begin{array}{l}
\frac{1}{{\Gamma (m - \alpha )}} \C{I}_t^{(\tau)} \left[{{{(t - \tau )}^{m - \alpha  - 1}}f^{(m)}}\right],\quad \alpha \notin \MBZP,\\
f^{(m)}(t),\quad \alpha \in \MBZP.
\end{array} \right.
\end{align}
The RL and Caputo FDs with sliding fixed memory length $L > 0$, denoted by $\MRLD{L}{t}{\alpha}{f(t)}$ and $\MCD{L}{t}{\alpha}{f(t)}$, respectively, are defined by
\begin{align}
\MRLD{L}{t}{\alpha}{f(t)} = \left\{ \begin{array}{l}
\frac{1}{{\Gamma (m - \alpha )}}\frac{{{d^m}}}{{d{t^m}}}\C{I}_{t-L, t}^{(\tau)}\left[{{{(t - \tau )}^{m - \alpha  - 1}}f}\right],\quad \alpha \notin \MBZP,\\
f^{(m)}(t),\quad \alpha \in \MBZP,
\end{array} \right.\\
\MCD{L}{t}{\alpha}{f(t)} = \left\{ \begin{array}{l}
\frac{1}{{\Gamma (m - \alpha )}} \C{I}_{t-L, t}^{(\tau)}\left[{{{(t - \tau )}^{m - \alpha  - 1}}f^{(m)}}\right],\quad \alpha \notin \MBZP,\label{eq:PMFCD1}\\
f^{(m)}(t),\quad \alpha \in \MBZP,
\end{array} \right.
\end{align}
cf. \cite{bourafa2021periodic}. If $f \in C^{(m)}(\MBRzerP)$, then $\MRLD{L}{t}{\alpha}{f(t)} = \MCD{L}{t}{\alpha}{f(t)}$, so we can denote both modified fractional operators by $\MD{L}{t}{\alpha}{}$. A reduced form of $\MD{L}{t}{\alpha}{}$ with constant integration limits, denoted by $\ED{L}{t}{\alpha}{f(t)}$, is given by
\begin{equation}\label{eq:RedElgPMFCD1}
\ED{L}{t}{\alpha}{f(t)} = \frac{L^{m-\alpha}}{(1-\alpha) \Gamma(m-\alpha)} \C{I}_1^{(y)} {\left[y^{\frac{m-1}{1-\alpha}} f^{(m)}\cancbra{t-L\,y^{\frac{1}{1-\alpha}}}\right]},
\end{equation}
cf. \cite{elgindy2023fourier}.

\begin{rem}
For a comprehensive review and background on the periodic derivative $\MD{L}{t}{\alpha}{f}$, the reader may consult \cite{bourafa2021periodic}. A proof that the modified derivative $\MD{L}{t}{\alpha}{f}$ indeed preserves the periodicity of $f \foralla$ periodic function $f \in C^{m}[a-L,b]: \alpha \in \MBRP\backslash\MBZP, L \in \MBRP, m = \left\lceil  \alpha  \right\rceil, \{a,b\} \subset \MBR: a < b$ can be found in \cite[Theorem 3.9]{bourafa2021periodic}. The reader may also consult Sections 3.4 and 3.5 of the same reference for a comparison between classical fractional-order derivatives and the modified derivative, in addition to two examples, including one on a physical model, to support the consistency of the modified derivative motif.
\end{rem}

\begin{rem}
The Gegenbauer polynomials we adopt here in this paper are the ones standardized by \citet[Eq. (6)]{Doha199075} or its equivalent form \cite[Eq. (A.1)]{elgindy2013fast}, where Chebyshev polynomials of the first kind and Legendre polynomials become special cases of this family of orthogonal polynomials for $\lambda = 0$ and $0.5$, respectively. The generated form of Gegenbauer polynomials are therefore different than those standardized by \citet{szeg1939orthogonal} in which the Gegenbauer polynomial evaluates to zero at $\lambda = 0$, for any nonnegative integer degree.
\end{rem}

\section{Problem Statement}
\label{sec:PS1}
In this study, we consider the following class of time-dependent one-dimensional FPDEs with variable coefficients and periodic solutions:
\begin{subequations}
\begin{equation}\label{eq:1}
a(x,t)\,\ED{L}{x}{\alpha}{u(x,t)} + b(x,t)\,\ED{L}{t}{\beta}{u(x,t)} = f(x,t),
\end{equation}
subject to the initial conditions
\begin{equation}\label{eq:2}
u(x,0) = g(x),\quad u(0,t) = h(t),
\end{equation}
\end{subequations}
where $\{a, b, f\} \subset C\left(\MBR^2\right), g \in {}^{1}_{T_1}\mathfrak{X}_{1}^{0}$, and $h \in {}^{1}_{T_2}\mathfrak{X}_{1}^{0}$ are some given functions, $u \in {}^{\quad 2}_{T_1,T_2}\mathfrak{X}_{1}^{1}$ is the solution, $\ED{L}{x}{\alpha}{u(x,t)}$ and $\ED{L}{t}{\beta}{u(x,t)}$ are the $\alpha$th- and $\beta$th-order RL/Caputo FDs of $u$ with sliding fixed memory length $L \in \MBRP$ with respect to $x$ and $t$, respectively, such that $\{\alpha,\beta\} \subset (0, 1]$.

\section{The FGPS Method}
\label{sec:FGPM1}
Collocating the FPDE \eqref{eq:1} at the rectangular mesh grid set $\MBS_{N_1,N_2}^{T_1,T_2}: \{N_1, N_2\} \subset \MBZeP$ provides the following $N_1 N_2$ system of linear equations:
\begin{equation}\label{eq:1}
a_{l,j}\,\ED{L}{x_{N_1,l}}{\alpha}{u(x,t_{N_2,j})} + b_{l,j}\,\ED{L}{t_{N_2,j}}{\beta}{u(x_{N_1,l},t)} = f_{l,j}\quad \forall (l,j) \in \MBJ_{N_1,N_2}.
\end{equation}
To solve Eqs. \eqref{eq:1} for the grid point values $u_{l,j}$, we can approximate the FDs by using \cite[Eq. (4.9)]{elgindy2023fourier} to obtain the following numerical formulas:
\begin{subequations}
\begin{align}
\ED{L}{x_{N_1,l}}{\alpha}{u(x,t_{N_2,j})} &\approx \sum\limits_{k = 0}^{N_1 - 1} {{}_{L}^E\C{D}^{\alpha}_{N_G,l,k} {u_{k,j}}}\quad \forall j \in \MBJ_{N_2},\label{eq:RedElgPMFCD101app2}\\
\ED{L}{t_{N_2,j}}{\beta}{u(x_{N_1,l},t)} &\approx \sum\limits_{k = 0}^{N_2 - 1} {{}_{L}^E\C{D}^{\beta}_{N_G,j,k} {u_{l,k}}}\quad \forall l \in \MBJ_{N_1},\label{eq:RedElgPMFCD101app22}
\end{align}
\end{subequations}
where 
\begin{equation}\label{eq:reqpr1}
{}_{L}^E\C{D}^{\gamma}_{N_G,r,s} = \frac{L^{1-\gamma}}{\Gamma(2-\gamma)} {}_{L}^E\C{Q}^{\gamma}_{N_G,r,s}\quad \forall \gamma \in (0,1),\,\{r,s\} \subset \MBJ_{n}: n \in \MBZeP,
\end{equation}
and ${}_{L}^E\C{Q}^{\gamma}_{N_G,r,s}$ is the $\gamma$th-order FG-based PS quadrature (FGPSQ) with index $L$ as defined by \cite[Formula (4.8)]{elgindy2023fourier}. We similarly refer to ${}_{L}^E\C{D}^{\gamma}_{N_G,r,s}$ by the $\gamma$th-order FG-based PS differentiation (FGPSD) with index $L$. Eqs. \eqref{eq:RedElgPMFCD101app2} and \eqref{eq:RedElgPMFCD101app22} can be further expressed in matrix notation as
\begin{subequations}
\begin{align}
\ED{L}{\bmx_{N_1}}{\alpha}{u(x,t_j)} &\approx {}_{L}^E\bsC{D}_{N_G}^{\alpha}\,u_{0:N_1-1,j},\label{eq:RedElgPMFCD101app3}\\
\ED{L}{\bmt_{N_2}}{\beta}{u(x_l,t)} &\approx {}_{L}^E\bsC{D}_{N_G}^{\beta}\,u_{l,0:N_2-1},\label{eq:RedElgPMFCD101app32}
\end{align}  
\end{subequations}
where 
\[
\ED{L}{\bmx_{N_1}}{\alpha}{u(x,t)} = \left[\ED{L}{x_{N_1,0}}{\alpha}{u(x,t)}, \ldots, \ED{L}{x_{N_1,N_1-1}}{\alpha}{u(x,t)}\right]^{\top},\quad \ED{L}{\bmt_{N_2}}{\beta}{u(x,t)} = \left[\ED{L}{t_{N_2,0}}{\beta}{u(x,t)}, \ldots, \ED{L}{t_{N_2,N_2-1}}{\beta}{u(x,t)}\right]^{\top},
\]
and
\begin{gather}
{}_{L}^E\bsC{D}^{\alpha}_{N_G} = \left[{}_{L}^E\bsC{D}^{\alpha}_{N_G,0}; \ldots; {}_{L}^E\bsC{D}^{\alpha}_{N_G,N_1-1}\right],\quad {}_{L}^E\bsC{D}^{\beta}_{N_G} = \left[{}_{L}^E\bsC{D}^{\beta}_{N_G,0}; \ldots; {}_{L}^E\bsC{D}^{\beta}_{N_G,N_2-1}\right]:\\
{}_{L}^E\bsC{D}^{\alpha}_{N_G,l} = \left[{}_{L}^E\C{D}^{\alpha}_{N_G,l,0}, \ldots, {}_{L}^E\C{D}^{\alpha}_{N_G,l,N_1-1}\right]\quad \forall l \in \MBJ_{N_1},\quad {}_{L}^E\bsC{D}^{\beta}_{N_G,j} = \left[{}_{L}^E\C{D}^{\beta}_{N_G,j,0}, \ldots, {}_{L}^E\C{D}^{\beta}_{N_G,j,N_2-1}\right]\quad \forall j \in \MBJ_{N_2}.
\end{gather}
We refer to ${}_{L}^E\bsC{D}^{\gamma}_{N_G}$ by the $\gamma$th-order FG-based PS fractional differentiation matrix (FGPSFDM) with index $L$. Substituting Formulas \eqref{eq:RedElgPMFCD101app2} and \eqref{eq:RedElgPMFCD101app22} into Eqs. \eqref{eq:1} and shuffling the terms that include the solution initial values onto the right hand side of
the equations yield the following approximate linear equations system:
\begin{equation}\label{eq:1K1}
a_{l,j}\,{}_{L}^E\bsC{D}^{\alpha}_{N_G,l,\hcancel{0}} u_{1:N_1-1,j} + b_{l,j}\,{}_{L}^E\bsC{D}^{\beta}_{N_G,j,\hcancel{0}} u_{l,1:N_2-1} = f_{l,j} - a_{l,j} {}_{L}^E\C{D}^{\alpha}_{N_G,l,0} h_j - b_{l,j}\,{}_{L}^E\C{D}^{\beta}_{N_G,j,0} g_{l}\quad \forall (l,j) \in \MBN_{N_1-1,N_2-1}.
\end{equation}
To put the pointwise representation of the derived system of Eqs. \eqref{eq:1K1} in the following matrix form 
\begin{equation}\label{eq:updatesys1}
\F{A}\,\bm{U} = \bm{F},
\end{equation}
we define the index matrix 
\begin{equation}\label{eq:smartindm1}
\C{N} = \F{X} + (N_1-1) \F{Y} + 1,
\end{equation}
where $\F{X}$ is the $(N_2-1) \times (N_1-1)$ matrix whose each row is a copy of the row array $[$0:$N_1$-2$]$, and $\F{Y}$ is the matrix of the same size with each column being a copy of the array $[$0:$N_2$-2$]$. The elements of the global collocation matrix $\F{A}$ and the column vector $\bm{F}$ are therefore given by
\begin{subequations}
	\begin{align}
	{\F{A}}\left({\C{N}(j,l),\C{N}(j,k)}\right) &= a_{l,j}\,{}_{L}^E\bsC{D}^{\alpha}_{N_G,l,k},\quad \forall k \in \{1:N_1-1\}_{\neq l},\label{eq:addeqeqA1}\\
	{\F{A}}\left({\C{N}(j,l),\C{N}(k,l)}\right) &= b_{l,j}\,{}_{L}^E\bsC{D}^{\beta}_{N_G,j,k},\quad \forall k \in \{1:N_2-1\}_{\neq j},\label{eq:addeqeqA2}\\
	{\F{A}}\left({\C{N}(j,l),\C{N}(j,l)}\right) &= a_{l,j}\,{}_{L}^E\bsC{D}^{\alpha}_{N_G,l,l} + b_{l,j}\,{}_{L}^E\bsC{D}^{\beta}_{N_G,j,j},\label{eq:addeqeqA3}\\
	\bm{F}\left({\C{N}(j,l)}\right) &= f_{l,j} - a_{l,j} {}_{L}^E\C{D}^{\alpha}_{N_G,l,0} h_j - b_{l,j}\,{}_{L}^E\C{D}^{\beta}_{N_G,j,0} g_{l},
\end{align}
$\forall (l,j) \in \MBN_{N_1-1,N_2-1}$.
\end{subequations}
Clearly, $\F{A}$ is a square matrix of size $(N_1-1) (N_2-1)$. It is interesting to note that $\F{A}$ is a sparse block matrix with the following special structure:
\begin{equation}
\F{A} = \left( {\begin{array}{*{20}{c}}
{\bs{\Lambda} _1^{\alpha ,\beta }}&{\bs{\Upsilon} _{1,2}^\beta }&{\bs{\Upsilon} _{1,3}^\beta }&{\bs{\Upsilon} _{1,4}^\beta }& \cdots &{\bs{\Upsilon} _{1,{N_2} - 1}^\beta }\\
{\bs{\Upsilon} _{2,1}^\beta }&{\bs{\Lambda} _2^{\alpha ,\beta }}&{\bs{\Upsilon} _{2,3}^\beta }&{\bs{\Upsilon} _{2,4}^\beta }& \cdots &{\bs{\Upsilon} _{2,{N_2} - 1}^\beta }\\
{\bs{\Upsilon} _{3,1}^\beta }&{\bs{\Upsilon} _{3,2}^\beta }&{\bs{\Lambda} _3^{\alpha ,\beta }}&{\bs{\Upsilon} _{3,4}^\beta }& \cdots &{\bs{\Upsilon} _{3,{N_2} - 1}^\beta }\\
 \vdots & \ddots & \ddots & \ddots & \ddots & \vdots \\
 \vdots & \ddots & \ddots & \ddots & \ddots &{\bs{\Upsilon} _{{N_2} - 2,{N_2} - 1}^\beta }\\
{\bs{\Upsilon} _{{N_2} - 1,1}^\beta }&{\bs{\Upsilon} _{{N_2} - 1,2}^\beta }&{\bs{\Upsilon} _{{N_2} - 1,3}^\beta }& \cdots &{\bs{\Upsilon} _{{N_2} - 1,{N_2} - 2}^\beta }&{\bs{\Lambda} _{{N_2} - 1}^{\alpha ,\beta }}
\end{array}} \right),
\end{equation}
where\\
\scalebox{0.8}{\parbox{\linewidth}{%
\begin{equation}
\bs{\Lambda} _k^{\alpha ,\beta } = \left( {\begin{array}{*{20}{c}}
{a_{1,k}\,{}_{L}^E\bsC{D}^{\alpha}_{N_G,1,1} + b_{1,k}\,{}_{L}^E\bsC{D}^{\beta}_{N_G,k,k}}&{a_{1,k}\,{}_{L}^E\bsC{D}^{\alpha}_{N_G,1,2}}& a_{1,k}\,{}_{L}^E\bsC{D}^{\alpha}_{N_G,1,3} & \cdots & {a_{1,k}\,{}_{L}^E\bsC{D}^{\alpha}_{N_G,1,N_1-1}}\\
{a_{2,k}\,{}_{L}^E\bsC{D}^{\alpha}_{N_G,2,1}}&{a_{2,k}\,{}_{L}^E\bsC{D}^{\alpha}_{N_G,2,2} + b_{2,k}\,{}_{L}^E\bsC{D}^{\beta}_{N_G,k,k}}&{a_{2,k}\,{}_{L}^E\bsC{D}^{\alpha}_{N_G,2,3}}& \cdots &{a_{2,k}\,{}_{L}^E\bsC{D}^{\alpha}_{N_G,2,N_1-1}}\\
 a_{3,k}\,{}_{L}^E\bsC{D}^{\alpha}_{N_G,3,1} &{a_{3,k}\,{}_{L}^E\bsC{D}^{\alpha}_{N_G,3,2}}&{\ddots}& \ddots & \vdots \\
 \vdots &{\ddots}&{\ddots}&{\ddots}&{a_{N_1-2,k}\,{}_{L}^E\bsC{D}^{\alpha}_{N_G,N_1-2,N_1-1}}\\
{a_{N_1-1,k}\,{}_{L}^E\bsC{D}^{\alpha}_{N_G,N_1-1,1}}&{a_{N_1-1,k}\,{}_{L}^E\bsC{D}^{\alpha}_{N_G,N_1-1,2}}&{\ldots}&{a_{N_1-1,k}\,{}_{L}^E\bsC{D}^{\alpha}_{N_G,N_1-1,N_1-2}}&{a_{N_1-1,k}\,{}_{L}^E\bsC{D}^{\alpha}_{N_G,N_1-1,N_1-1} + b_{N_1-1,k}\,{}_{L}^E\bsC{D}^{\beta}_{N_G,k,k}}
\end{array}} \right),
\end{equation}
}}\\

and
\begin{equation}
\bs{\Upsilon} _{k,s}^\beta  = \left( {\begin{array}{*{20}{c}}
{b_{1,k}\,{}_{L}^E\bsC{D}^{\beta}_{N_G,k,s}}&0& \cdots &0\\
0&{b_{2,k}\,{}_{L}^E\bsC{D}^{\beta}_{N_G,k,s}}& \ddots & \vdots \\
 \vdots & \ddots & \ddots &0\\
0& \cdots &0&{b_{N_1-1,k}\,{}_{L}^E\bsC{D}^{\beta}_{N_G,k,s}}
\end{array}} \right)\quad \forall \{k,s\} \subset \MBN_{N_2-1}.
\end{equation}
We can solve the linear system \eqref{eq:updatesys1} for the approximate solution values $\tilde u_{1:N_1-1,1:N_2-1}$ by a direct solver
through (a variant of) Gauss elimination for sufficiently smooth variable coefficients, source functions, and initial value functions owing to the exponential convergence of the FGPS approximations, as we shall demonstrate later in Section \ref{sec:NS1}. We can further estimate the approximate solution function $I_{N_1,N_2}\tilde u(x,t)$ at any point $(x,t) \in \MBR^2$ using the following Fourier interpolation formula in Lagrange form:
\begin{equation}\label{eq:627}
I_{N_1,N_2}\tilde u(x,t) = \sum\limits_{l = 0}^{N_1-1} {\sum\limits_{j = 0}^{N_2-1} {\tilde u_{l,j} \C{F}_{l,j}^{x,t}(x,t)}},
\end{equation}
where $\C{F}_{l,j}^{x,t}(x,t) = {\C{F}}_l^x(x) {\C{F}}_j^t(t)$ is the tensor product trigonometric Lagrange interpolating polynomial such that\\ 
\scalebox{0.95}{\parbox{\linewidth}{%
\begin{align}
&{\C{F}}_l^x(x) = \frac{1}{N_1}\sumd\sum\limits_{\left| k \right| \le N_1/2} {e^{i {\omega^x _k}(x - {x_{N_1,l}})}} = \left\{ \begin{array}{l}
1,\quad x = x_{N_1,l},\nonumber\\
\frac{1}{N_1}\sin(N_1 \nu^x_l)\cot({\nu^x_l}),\quad x \ne x_{N_1,l},
\end{array} \right.:\quad \omega^x_k = \frac{2 \pi k}{T_1}\,\forall k \in \MFK'_{N_1},\quad \nu^x_l = \frac{\pi}{T_1} \left( {x - {x_{N_1,l}}} \right)\quad \forall l \in \MBJ_{N_1},\\
&{\C{F}}_j^t(t) = \frac{1}{N_2}\sumd\sum\limits_{\left| k \right| \le N_2/2} {e^{i {\omega^t _k}(t - {t_{N_2,j}})}} = \left\{ \begin{array}{l}
1,\quad t = t_{N_2,j},\nonumber\\
\frac{1}{N_2}\sin(N_2 \nu^t_j)\cot({\nu^t_j}),\quad t \ne t_{N_2,j},
\end{array} \right.:\quad \omega^t_k = \frac{2 \pi k}{T_2}\,\forall k \in \MFK'_{N_2},\quad \nu^t_j = \frac{\pi}{T_2} \left( {t - {t_{N_2,j}}} \right)\quad \forall j \in \MBJ_{N_2}.
\end{align}
}}

\section{Computational Storage}
\label{sec:CS1}
In this section, we briefly discuss the computational storage necessary to set up the linear system \eqref{eq:updatesys1}. We determine first the computational storage required by the $\gamma$th-order FGPSFDM with index $L$, ${}_{L}^E\bsC{D}^{\gamma}_{N_G}$, which follows from the following corollary.

\begin{cor}\label{cor:Toeplitz1}
The $\gamma$th-order FGPSFDM ${}_{L}^E\bsC{D}^{\gamma}_{N_G}$ is a Toeplitz matrix.
\end{cor}
\begin{proof}
Let ${}_{L}^E\bsC{Q}^{\gamma}_{N_G}$ be the $\gamma$th-order Fourier-Gegenbauer-based PS integration matrix with index $L$ defined by
\begin{equation}
{}_{L}^E\bsC{Q}^{\gamma}_{N_G} = \left[{}_{L}^E\bsC{Q}^{\gamma}_{N_G,0}; \ldots; {}_{L}^E\bsC{Q}^{\gamma}_{N_G,N-1}\right]:\quad {}_{L}^E\bsC{Q}^{\gamma}_{N_G,l} = \left[{}_{L}^E\C{Q}^{\gamma}_{N_G,l,0}, \ldots, {}_{L}^E\C{Q}^{\gamma}_{N_G,l,N-1}\right]\quad \forall l \in \MBJ_N.
\end{equation}
The proof follows immediately by realizing that ${}_{L}^E\bsC{Q}^{\gamma}_{N_G}$ is a Toeplitz matrix as proven by \cite[Theorem 4.1]{elgindy2023fourier}, and 
\begin{equation}
{}_{L}^E\bsC{D}^{\gamma}_{N_G} = \frac{L^{1-\gamma}}{\Gamma(2-\gamma)} {}_{L}^E\bsC{Q}^{\gamma}_{N_G}\quad \forall \gamma \in (0,1),
\end{equation}
by Eq. \eqref{eq:reqpr1}.
\end{proof}
Since ${}_{L}^E\bsC{D}^{\gamma}_{N_G}$ is a square matrix of size $N$, Corollary \ref{cor:Toeplitz1} manifests that ${}_{L}^E\bsC{D}^{\gamma}_{N_G}$ requires only $2N - 1$ storage. Therefore, one can solve Toeplitz systems of the form ${}_{L}^E\bsC{D}^{\gamma}_{N_G} \bmx = \bmb\,\foralls \{\bmx, \bmb\} \subset \MBR^N$ using only $O(N^2)$ flops and $O(N)$ storage with the aid of special fast algorithms such as the Levinson-Durbin algorithm \cite{bjorck2008numerical}. The $\gamma$th-order FGPSFDM with index $L$, ${}_{L}^E\bsC{D}^{\gamma}_{N_G}$, is a constant matrix that can be constructed and stored offline for a certain range of its parameters and invoked later when running the numerical program. If we now turn our attention to the global coefficient matrix $\F{A}$ of the linear system \eqref{eq:updatesys1}, we can clearly see from Section \ref{sec:FGPM1} that it is a block matrix with $N_2-1$ dense diagonal square matrices and $(N_1 - 2) (N_2 - 1)$ diagonal matrices, each of size $N_1-1$. Hence, the matrix $\F{A}$ requires $(N_2-1) (2N_1^2-5 N_1+3)$ storage. When we add this to the $(N_1-1) (N_2-1)$ storage requirements of the column vector $\bm{F}$, we find out that the total storage requirements of the linear system \eqref{eq:updatesys1} is $2\,(N_1 -1)^2 \,(N_2 -1)$. In the special case when $N_1 = N_2 = N$, the computational storage requirement of the linear system simplifies into $2 (N-1)^3$. We shall discover in the next section that the proposed method can often achieve superior accuracy when both $N_1$ and $N_2$ are relatively very small, thus, significantly reducing the computational storage requirement.

\section{Error and Convergence Analysis}
\label{sec:ECA1}
The following theorem underlines the truncation error bounds of the approximate linear system of equations \eqref{eq:1K1}.
\begin{thm}\label{thm:Jan212022}
Let $N_G \in \MBZP, L \in \MBRP, \gamma  = \min \left\{ {\alpha ,\beta } \right\}, (x,t) \in \MBOmega_{T_1,T_2}$, and suppose that ${{\C{F}'_l}^x}\left(x-L\,y^{\frac{1}{1-\alpha}}\right)$ and ${{\C{F}'_j}^t}\left(t-L\,y^{\frac{1}{1-\beta}}\right)$ are approximated by the SGG interpolants obtained through interpolation at the SGG points set $\MBSG_{1,N_G}^{\lambda}\,\forall (l,j) \in \MBJ_{N_1,N_2}$; cf. \cite{Elgindy20161,Elgindy2023a,elgindy2023fourier}. If \begin{equation}\label{eq:tss1}
L > 1-\gamma,
\end{equation}
then the truncation error of the approximate linear system of equations \eqref{eq:1K1} at each collocation mesh grid is bounded by
\begin{equation}\label{eq:errhiheq1}
\left| {E_{l,j}^{\text{tr}}} \right| \le \frac{{2\mu {L^{1 - \gamma }}}}{{\Gamma (2 - \gamma )}}\left( {{N_{\max }}\left\|u\right\|_{\infty}{E^{{\text{fd}}}} + {E^{\text{di}}_{{N_1},{N_2}}}} \right),
\end{equation}
where 
$\mu  = \max \left\{ {{{\left\| a \right\|}_\infty },{{\left\| b \right\|}_\infty }} \right\}, {N_{\max }} = \max \left\{ {{N_{1:2}}} \right\}, E^{\text{fd}} = \max\left\{E_{N_1,N_G}^{\lambda,\alpha},E_{N_2,N_G}^{\lambda,\beta}\right\}: \foralle N \in \{N_1,N_2\}, \nu \in \{\alpha,\beta\}$,\\ $E_{N,N_G}^{\lambda,\nu} = \mathop {\max }\limits_{(k,s) \subset \MBJ_N}\left|E_{N,N_G,k}^{\lambda,\nu}(\zeta_{k,s})\right|:$\\
\scalebox{0.9}{\parbox{\linewidth}{%
\begin{equation}
	\left|{E_{N,N_G,k}^{\lambda,\nu}(\zeta_{k,s})}\right| \le  B_{2,N,N_G}^{L,\nu,\zeta_{k,s}}\,{2^{ - 2{N_G} - 1}}{{{e}}^{{N_G}}}{N_G}^{\lambda - {N_G} - \frac{3}{2}} {\left\{ \begin{array}{l}
	1,\quad {N_G} \ge 0 \wedge \lambda \ge 0,\\
	\displaystyle{\frac{{\Gamma \left( {\frac{{{N_G}}}{2} + 1} \right)\Gamma \left( {\lambda + \frac{1}{2}} \right)}}{{\sqrt \pi\,\Gamma \left( {\frac{{{N_G}}}{2} + \lambda + 1} \right)}}},\quad N_G \in \MBZOP \wedge  - \frac{1}{2} < \lambda < 0,\\
	\displaystyle{\frac{{2\Gamma \left( {\frac{{{N_G} + 3}}{2}} \right)\Gamma \left( {\lambda + \frac{1}{2}} \right)}}{{\sqrt \pi  \sqrt {\left( {{N_G} + 1} \right)\left( {{N_G} + 2\lambda + 1} \right)}\,\Gamma \left( {\frac{{{N_G} + 1}}{2} + \lambda} \right)}}},\quad N_G \in \MBZzereP \wedge  - \frac{1}{2} < \lambda < 0,\\
	B_1^{\lambda} {\left( {{N_G} + 1} \right)^{ - \lambda}},\quad {N_G} \to \infty  \wedge  - \frac{1}{2} < \lambda < 0,
	\end{array} \right.}\label{eq:wow1}
\end{equation}
}}\\
where $B_{2,N,N_G}^{L,\nu,\zeta_{k,s}} = {D^{\lambda}} {A_{N,N_G}^{L,\nu,\zeta_{k,s}}}:$\\ 
\scalebox{0.9}{\parbox{\linewidth}{%
\begin{equation}\label{eq:Remar1}
{A_{N,N_G}^{L,\nu,\zeta_{k,s}}} = N^{N_G+1} \zeta_{k,s}^{-(N_G+1)} \left(\frac{L}{1-\nu}\right)^{N_G+1} \gamma_{N_G}^{\nu} \left\{ \begin{array}{l}
c_1 \left[ {\left( {\begin{array}{*{20}{c}}
{{N_G} + 1}\\
{{N_G}/2}
\end{array}} \right){\delta _{\frac{{{N_G}}}{2},\left\lfloor {\frac{{{N_G}}}{2}} \right\rfloor }} + \left( {\begin{array}{*{20}{c}}
{{N_G} + 1}\\
{\left\lfloor {{N_G}/2} \right\rfloor }
\end{array}} \right)\left( {1 - {\delta _{\frac{{{N_G}}}{2},\left\lfloor {\frac{{{N_G}}}{2}} \right\rfloor }}} \right)} \right],\quad N_G \in \MBZP,\\
c_2 \frac{(2 e)^{N_G/2}}{\sqrt{N_G}},\quad N_G \in\,\oset{\rightarrow}{\MBZeP},\\
c_3 \frac{(e N_G)^{\left\lfloor {N_G/2} \right\rfloor}}{{\left( {\left\lfloor {N_G/2} \right\rfloor } \right)^{\left\lfloor {N_G/2} \right\rfloor  + 1/2}}},\quad N_G \in \,\oset{\rightarrow}{\MBZOP},
\end{array} \right.
\end{equation}
}}\\
$\foralls \{c_{1:3}\} \subset \MBRP, \{\zeta_{0:N-1,0:N-1}\} \subset (0,1), \lambda$-dependent constants $\{{D^{\lambda}},B_1^{\lambda}\} \subset \MBRP: B_1^{\lambda} > 1, \gamma_{N_G}^{\nu} = \sum\nolimits_{k \in {\MBJ_{{N_G} + 2}}} {\left| {{{\left( {\frac{\nu }{{1 - \nu }}} \right)}^{({N_G} - k + 1)}}} \right|}$, and ${E^{\text{di}}_{{N_1},{N_2}}}$ is the maximum of the absolute derivatives of Fourier interpolation errors of $u$ along the coordinate axes of the $xt$-plane and based on the mesh grid sets $S_{N_1}^{T_1}$ and $S_{N_2}^{T_2}$.   
\end{thm}
\begin{proof}
First notice that ${L^{1 - \gamma }}/{\Gamma (2 - \gamma )}$ is monotonically decreasing on $\MBOmega_1$ such that 
\[\gamma = \mathop {\argmax}\limits_{\nu} \frac{L^{1 - \nu}}{\Gamma (2 - \nu)},\]
cf. \cite[Theorem 5.3]{elgindy2023fourier}. Let ${I_{N_1}}u(x,t_{N_2,j}) = {\C{F}}_{0:N_1-1}^x\cancbra{x}\, u_{x_{N_1,0:N_1-1},t_{N_2,j}}$ be the $N_1/2$-degree, $T_1$-periodic Fourier interpolant that matches $u(x,t_{N_2,j})$ at the mesh points set $\MBS_{N_1}^{T_1}\,\foralle j \in \MBJ_{N_2}$. Let also ${}_xE^{\text{int}}_{N_1,j}$ denote the associated interpolation error such that
\begin{equation}
u(x,t_{N_2,j}) = I_{N_1}u(x,t_{N_2,j}) + {}_xE^{\text{int}}_{N_1,j}(x)\quad \forall j \in \MBJ_{N_2}.
\end{equation}
Taking the periodic FD operator $\ED{L}{x}{\alpha}{}$ of both sides of the equation and following the work of \citet{elgindy2023fourier} to numerically compute the FD at each spatial mesh grid point yield
\begin{align}
\ED{L}{x_{N_1,l}}{\alpha}{u(x,t_{N_2,j})} &= \frac{L^{1-\alpha}}{\Gamma(2-\alpha)} \left(\C{I}_1^{(y)}{\C{F}'}_{0:N_1-1}^x\cancbra{x_{N_1,l}-Ly^{\frac{1}{1-\alpha}}}\, u_{x_{N_1,0:N_1-1},t_{N_2,j}} + \C{I}_1^{(y)}{{}_xE'^{\text{int}}_{N_1,j}}\cancbra{x_{N_1,l}-Ly^{\frac{1}{1-\alpha}}}\right)\\
&= \frac{L^{1-\alpha}}{\Gamma(2-\alpha)} \left[\sum_{k \in \MBJ_{N_1}} {u(x_{N_1,k},t_{N_2,j}) \left({}_{L}^E\C{Q}^{\alpha}_{N_G,l,k} + E_{N_1,N_G,k}^{\lambda,\alpha}\left(\zeta_{k,l}\right)\right)} + \C{I}_1^{(y)}{{}_xE'^{\text{int}}_{N_1,j}}\cancbra{x_{N_1,l}-Ly^{\frac{1}{1-\alpha}}}\right],\label{eq:helkk1}
\end{align}
$\foralls \{\zeta_{0:N_1-1,l}\} \subset (0,1)$, where $E_{N_1,N_G,k}^{\lambda,\alpha}$ is the FGPSQ error as defined by \cite[Theorem 1]{elgindy2023fourier} $\forall k \in \MBJ_{N_1}$. By rearranging the terms in Eq. \eqref{eq:helkk1} we obtain
\begin{align}
\ED{L}{x_{N_1,l}}{\alpha}{u(x,t_{N_2,j})} &= {}_{L}^E\bsC{D}^{\alpha}_{N_G,l}\, u_{x_{N_1,0:N_1-1},t_{N_2,j}} + \frac{L^{1-\alpha}}{\Gamma(2-\alpha)} \left[\sum_{k \in \MBJ_{N_1}} {u(x_{N_1,k},t_{N_2,j}) E_{N_1,N_G,k}^{\lambda,\alpha}\left(\zeta_{k,l}\right)} + \C{I}_1^{(y)}{{}_xE'^{\text{int}}_{N_1,j}}\cancbra{x_{N_1,l}-Ly^{\frac{1}{1-\alpha}}}\right].
\end{align}
Therefore,
\begin{subequations}
\begin{equation}\label{eq:res1}
\left|\ED{L}{x_{N_1,l}}{\alpha}{u(x,t_{N_2,j})} - {}_{L}^E\bsC{D}^{\alpha}_{N_G,l}\, u_{x_{N_1,0:N_1-1},t_{N_2,j}}\right| \le \frac{L^{1-\alpha}}{\Gamma(2-\alpha)} \left[N_1 \left\|u\right\|_{\infty} E_{N_1,N_G}^{\lambda,\alpha} + \left\|{}_xE'^{\text{int}}_{N_1}\right\|_{\infty}\right],
\end{equation}
where $\left\|{}_xE'^{\text{int}}_{N_1}\right\|_{\infty} = \mathop {\max }\limits_{j \in \MBJ_{N_2}} \left\|{}_xE'^{\text{int}}_{N_1,j}\right\|_{\infty}$. By similarity, we can readily show that 
\begin{equation}\label{eq:res2}
\left|\ED{L}{t_{N_2,j}}{\beta}{u(x_{N_1,l},t)} - {}_{L}^E\bsC{D}^{\beta}_{N_G,j}\, u_{x_{N_1,l},t_{N_2,0:N_2-1}}\right| \le \frac{L^{1-\beta}}{\Gamma(2-\beta)} \left[N_2 \left\|u\right\|_{\infty} E_{N_2,N_G}^{\lambda,\beta} + \left\|{}_tE'^{\text{int}}_{N_2}\right\|_{\infty}\right],
\end{equation}
\end{subequations}
where $\left\|{}_tE'^{\text{int}}_{N_2}\right\|_{\infty} = \mathop {\max }\limits_{l \in \MBJ_{N_1}} \left\|{}_tE'^{\text{int}}_{N_2,l}\right\|_{\infty}$, and ${}_tE^{\text{int}}_{N_2,l}$ is the associated interpolation error of $u(x_{N_1,l},t)$ based on the mesh points set $S_{N_2}^{T_2}\,\forall l \in \MBJ_{N_1}$. Now, let ${E^{\text{di}}_{{N_1},{N_2}}} = \max\left\{\left\|{}_xE'^{\text{int}}_{N_1}\right\|_{\infty}, \left\|{}_tE'^{\text{int}}_{N_2}\right\|_{\infty}\right\}$. The proof is established by subtracting Eqs. \eqref{eq:1} and \eqref{eq:1K1}, imposing Ineqs. \eqref{eq:res1} and \eqref{eq:res2}, and applying \cite[Theorem 5.2]{elgindy2023fourier}.
\end{proof}
The first term in the sum of Ineq. \eqref{eq:errhiheq1} is due to the FGPS approximations to the FDs of $u$, hence can be considered pure FD computational (FDC) error. The second term is due to the Fourier interpolations of $u$ at the collocation mesh grid and thus can be viewed as a pure Fourier-interpolation-induced (FII) error. It is interesting to note that increasing the maximum degree of Fourier interpolants of $u$ generally decreases the FII error, but increases the FDC error. This observation agrees with the recent findings of \cite{elgindy2023fourier} on FDC errors based on FGPS approximations. Fortunately, both FDC and FII errors can be made very small for sufficiently smooth periodic solutions using relatively small mesh grids, owing to the rapid convergence of FGPS approximations and Fourier interpolation.

\begin{rem}
The reader should realize that the FGPS approximations of the periodic FDs converge exponentially fast for sufficiently smooth periodic functions, as
$N_G \to \infty$, while holding all other parameters fixed; cf. \cite[Theorem 5.2 and the paragraph that follows]{elgindy2023fourier}.
\end{rem}

The following corollary is a direct result of \cite[Corollary 5.1]{Elgindy2023b}, which states that the error of the FGPS method is dominated by the FDC error when the solution of the FPDE is ``$\beta$-analytic'' owing to the collapse of the FII error; cf. \cite{Elgindy2023a}.
\begin{cor}\label{cor:1}
Let the assumptions of Theorem \ref{thm:Jan212022} hold. If $u$ is $\beta$-analytic with respect to both arguments, then ${E^{\text{di}}_{{N_1},{N_2}}}$ collapses, and Ineq. \eqref{eq:errhiheq1} reduces to
\begin{equation}\label{eq:errhiheq1mz1}
\left| {E_{l,j}^{\text{tr}}} \right| \le \frac{{2 \mu {N_{\max }} {L^{1 - \gamma }} \left\|u\right\|_{\infty}}}{{\Gamma (2 - \gamma )}} {{E^{{\text{fd}}}}}.
\end{equation}
\end{cor}

\begin{rem}
We draw the reader’s attention to the fact that $\beta$ in the expression ``$\beta$-analytic'' has nothing to do with the order of the FD, $\ED{L}{t}{\beta}{}$, often used in the presentation of the paper, but rather pertained to a type of analyticity of functions; the reader may consult \cite{Elgindy2023a} for more details.
\end{rem}

\section{Numerical Simulations}
\label{sec:NS1}
This section demonstrates the accuracy and performance of the proposed FGPS method on four novel test problems. All numerical experiments were carried out using MATLAB R2023a software installed on a personal laptop equipped with a 2.9 GHz AMD Ryzen 7 4800H CPU and 16 GB memory running on a 64-bit Windows 11 operating system. The linear system of equations \eqref{eq:updatesys1} was solved using MATLAB mldivide solver. The FGPS was performed in all cases using the parameter values $N_1 = N_2 = 4, L = 30, N_G = 1000$, and $\lambda = 0$, except when otherwise explicitly stated.\medskip

\textbf{Problem 1.} Consider the following time-dependent one-dimensional FPDE with variable coefficients and periodic solutions:
\begin{subequations}
\begin{equation}\label{eq:1new1}
x\,t\,\ED{L}{x}{1/2}{u(x,t)} + (x+t)\,\ED{L}{t}{1/2}{u(x,t)} = \frac{1}{2}\sqrt[4]{{ - 1}}\left[ {{e^{ - it}}\left( {t + x} \right)\sin \left( x \right)\left( {\erf\left( c \right){e^{2it}} - \erfi\left( c \right)} \right) - it{e^{ - ix}}x\cos \left( t \right)\left( {\erf\left( c \right){e^{2ix}} + \erfi\left( c \right)} \right)} \right],
\end{equation}
subject to the initial conditions
\begin{equation}\label{eq:2new1}
u(x,0) = \sin(x),\quad u(0,t) = 0,
\end{equation}
\end{subequations}
$\forall (x,t) \in \MBOmega_{2\pi, 2\pi}$, where $c = \sqrt {15} \left( {1 + i} \right)$. The exact solution of the problem is $u(x,t) = \sin x \cos t$. Figure \ref{fig:Fig1} shows the plots of the exact and approximate solution obtained by the FGPS method. The figure also shows the corresponding absolute error surface, which indicates that the solution is resolved to about machine precision. $\kappa(\F{A})$ and the elapsed time for running the FGPS method were about $12.615$ and $0.064$ s, respectively, rounded to three decimal digits. Figure \ref{fig:Fig12} demonstrates that one can still retain highly accurate approximations using relatively much smaller values of $N_G$. In addition, Figure \ref{fig:Fig13} supports our theoretical error and convergence analysis study discussed in Section \ref{sec:ECA1}, where we observe a decline in the approximations precision when increasing the degrees of Fourier interpolants of $u$ by $16$ orders due to the rise of the FDC error.

\begin{figure}[H]
\centering
\includegraphics[scale=0.4]{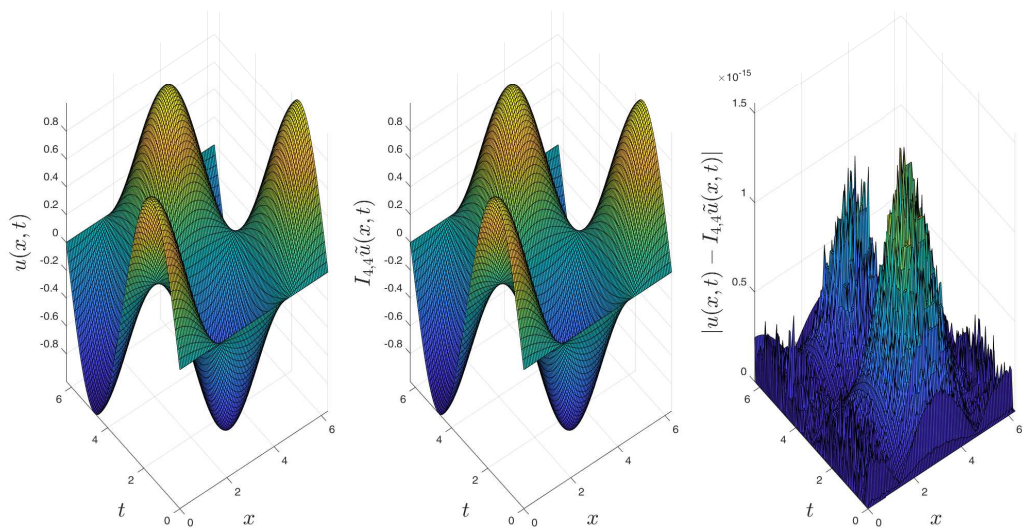}
\caption{The plots of the exact (left), approximate solution obtained by the FGPS method (middle), and the corresponding absolute errors (right) for Problem 1. The plots were generated using $100$ equally-spaced nodes in the $x$- and $t$-directions from $0$ to $2 \pi$.}
\label{fig:Fig1}
\end{figure}

\begin{figure}[H]
\centering
\includegraphics[scale=0.4]{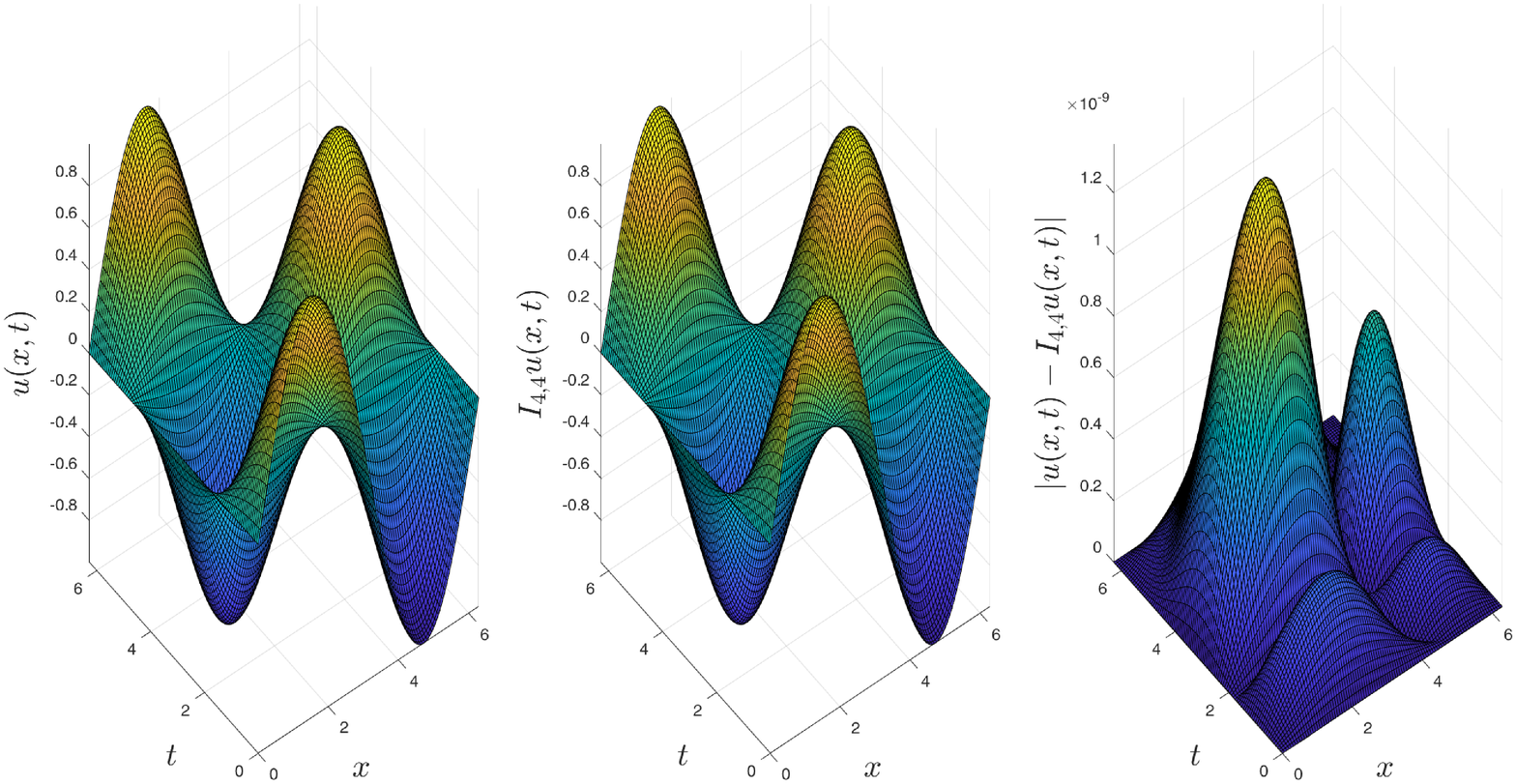}
\caption{The plots of the exact (left), approximate solution obtained by the FGPS method (middle), and the corresponding absolute errors (right) for Problem 1 using the parameter value $N_G = 40$. The plots were generated using $100$ equally-spaced nodes in the $x$- and $t$-directions from $0$ to $2 \pi$.}
\label{fig:Fig12}
\end{figure}

\begin{figure}[H]
\centering
\includegraphics[scale=0.4]{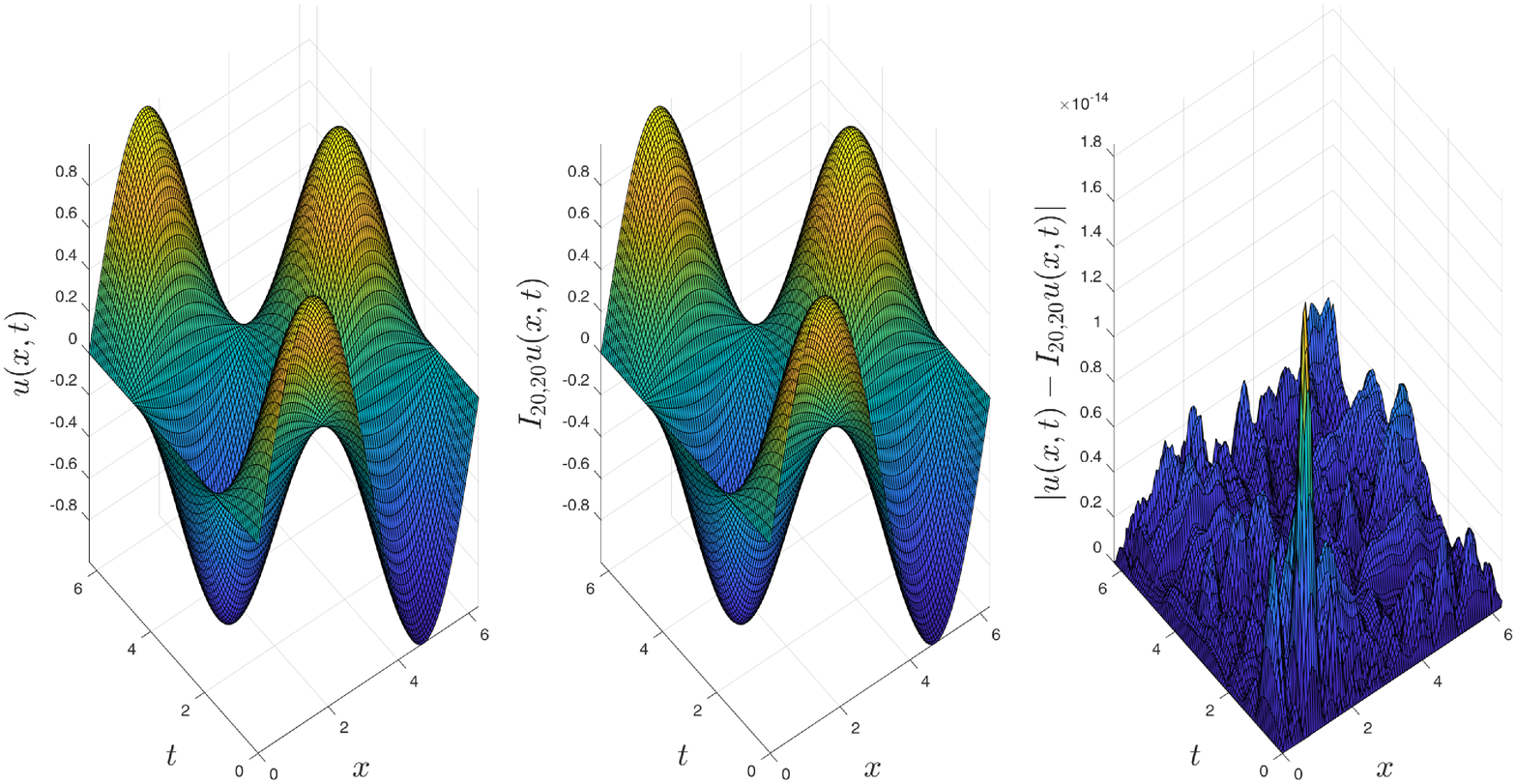}
\caption{The plots of the exact (left), approximate solution obtained by the FGPS method (middle), and the corresponding absolute errors (right) for Problem 1 using the parameter values $N_1 = N_2 = 20$. The plots were generated using $100$ equally-spaced nodes in the $x$- and $t$-directions from $0$ to $2 \pi$.}
\label{fig:Fig13}
\end{figure}

\textbf{Problem 2.} Consider the following time-dependent one-dimensional FPDE with variable coefficients and periodic solutions:
\begin{subequations}
\begin{align}
\sin(x\,t)\,\ED{L}{x}{1/3}{u(x,t)} &+ \cos\left(x+t^2\right)\,\ED{L}{t}{2/3}{u(x,t)} = \frac{{\sqrt[6]{{ - 1}}{e^{ - it}}\cos \left( {x + {t^2}} \right)}}{{2\Gamma \left( {\frac{1}{3}} \right)}}\left[ {{{( - 1)}^{2/3}}{e^{2it}}\left( {\Gamma \left( {\frac{1}{3}} \right) - \Gamma \left( {\frac{1}{3},30i} \right)} \right) - \Gamma \left( {\frac{1}{3}} \right) + \Gamma \left( {\frac{1}{3}, - 30i} \right)} \right]\nonumber\\
&+ \frac{{\sqrt[3]{3}{e^{ - i(3x + 1)}}\sin (xt)}}{{4\Gamma \left( {\frac{2}{3}} \right)}}\left[ {\left( {\sqrt 3  + i} \right){e^{6ix + 2i}}\left( {\Gamma \left( {\frac{2}{3}} \right) - \Gamma \left( {\frac{2}{3},90i} \right)} \right) + \left( {\sqrt 3  - i} \right)\Gamma \left( {\frac{2}{3}} \right) + 2{{( - 1)}^{5/6}}\Gamma \left( {\frac{2}{3}, - 90i} \right)} \right],\label{eq:1new1}
\end{align}
subject to the initial conditions
\begin{equation}\label{eq:2new1}
u(x,0) = \cos(3x+1),\quad u(0,t) = \cos 1 - \sin t,
\end{equation}
\end{subequations}
$\forall (x,t) \in \MBOmega_{2\pi/3, 2\pi}$. The exact solution of the problem is $u(x,t) = \cos(3x+1) - \sin t$. The numerical results are shown in Figure \ref{fig:Fig2}. $\kappa(\F{A})$ and the elapsed time for running the FGPS method were about $36.108$ and $0.093$ s, respectively, rounded to three decimal digits. 

\begin{figure}[H]
\centering
\includegraphics[scale=0.4]{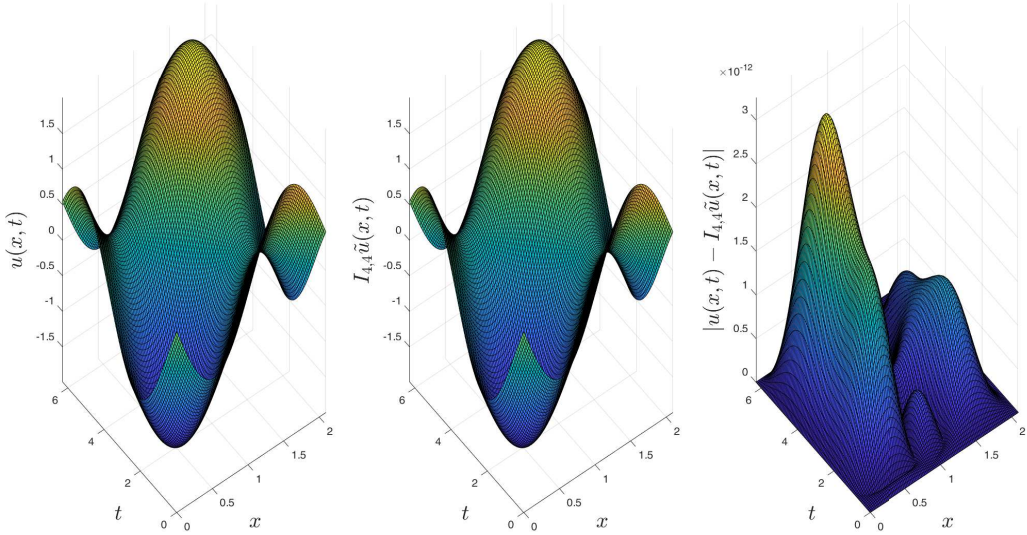}
\caption{The plots of the exact (left), approximate solution obtained by the FGPS method (middle), and the corresponding absolute errors (right) for Problem 2. The plots were generated using $100$ equally-spaced nodes in the $x$- and $t$-directions from $(0,0)$ to $(2 \pi/3, 2 \pi)$.}
\label{fig:Fig2}
\end{figure}

\textbf{Problem 3.} Consider the following time-dependent one-dimensional FPDE with variable coefficients and periodic solutions:
\begin{subequations}
\begin{align}
&e^{-x\,t}\,\ED{L}{x}{7/10}{u(x,t)} + \ln(x-t+3 \pi)\,\ED{L}{t}{4/5}{u(x,t)} = \frac{{{\left( { - 1} \right)}^{3/20}}{{{e}}^{ - \left( {\left( {t + 2{{i}}} \right)x} \right)}}\sin \left( t \right)}{{{2^{3/10}}\Gamma \left( {\frac{3}{{10}}} \right)}}\left[ {{\left( { - 1} \right)}^{7/10}}{{{e}}^{4{{i}}x}}\left( { - \Gamma \left( {\frac{3}{{10}}} \right) + \Gamma \left( {\frac{3}{{10}},60{{i}}} \right)} \right)\right.\\
&\left.+ \Gamma \left( {\frac{3}{{10}}} \right) - \Gamma \left( {\frac{3}{{10}}, - 60{{i}}} \right) \right] + \frac{\sqrt[{10}]{{ - 1}}{{{e}}^{ - {{i}}t}}\sin \left( {2x} \right)\ln \left( { - t + x + 3\pi } \right)}{{2\Gamma \left( {\frac{1}{5}} \right)}} \left[ {{{\left( { - 1} \right)}^{4/5}}{{{e}}^{2{{i}}t}}\left( { - \Gamma \left( {\frac{1}{5}} \right) + \Gamma \left( {\frac{1}{5},30{{i}}} \right)} \right) + \Gamma \left( {\frac{1}{5}} \right) - \Gamma \left( {\frac{1}{5}, - 30{{i}}} \right)} \right],\label{eq:1new1}
\end{align}
subject to the initial conditions
\begin{equation}\label{eq:2new1}
u(x,0) = u(0,t) = 0,
\end{equation}
\end{subequations}
$\forall (x,t) \in \MBOmega_{\pi, 2\pi}$. The exact solution of the problem is $u(x,t) = \sin(2 x) \sin(t)$. The numerical results are shown in Figure \ref{fig:Fig3}. $\kappa(\F{A})$ and the elapsed time for running the FGPS method were about $3.253$ and $0.093$ s, respectively, rounded to three decimal digits. 

\begin{figure}[H]
\centering
\includegraphics[scale=0.4]{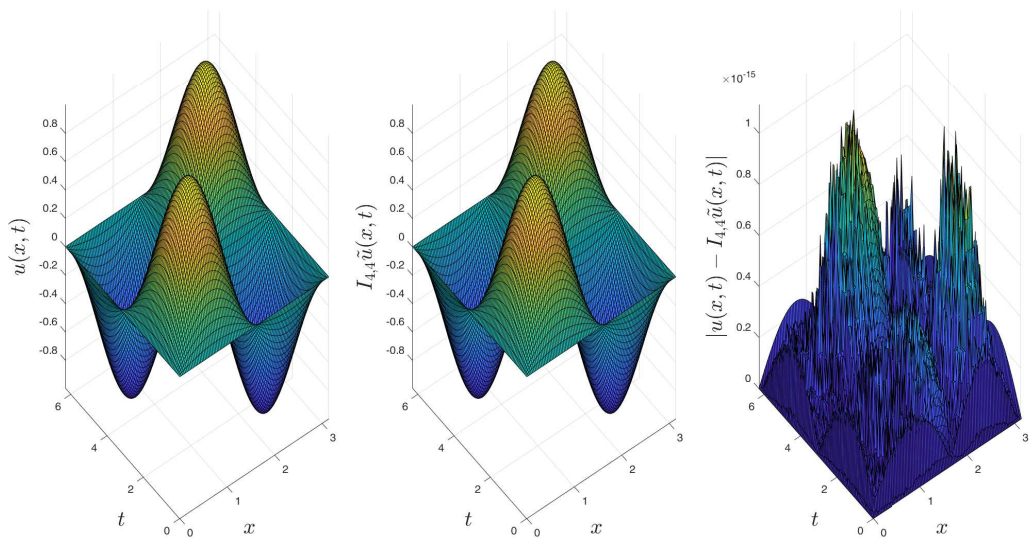}
\caption{The plots of the exact (left), approximate solution obtained by the FGPS method (middle), and the corresponding absolute errors (right) for Problem 3. The plots were generated using $100$ equally-spaced nodes in the $x$- and $t$-directions from $(0,0)$ to $(\pi, 2 \pi)$.}
\label{fig:Fig3}
\end{figure}

\textbf{Problem 4.} Consider the following time-dependent one-dimensional FPDE with variable coefficients and periodic solutions:
\begin{subequations}
\begin{align}
&(x+5)\,\ED{L}{x}{\alpha}{u(x,t)} - x t^2\,\ED{L}{t}{\beta}{u(x,t)} = \left[ {\left( {{t^2} - 1} \right)x - 5} \right]\sin (x + t) + \sinh (0.3)\left[ {(x + 5)\cos (x)\sin (t) - x\,{t^2}\sin (x)\cos (t)} \right],\label{eq:1new102May}
\end{align}
subject to the initial conditions
\begin{equation}\label{eq:2new1}
u(x,0) = \cos x,\quad u(0,t) = \cos t,
\end{equation}
\end{subequations}
$\forall (x,t) \in \MBOmega_{2\pi, 2\pi}$. The exact solution of the problem for $\alpha = \beta = 1$ is $u^{1,1}(x,t) = \cos (x + t) + \sinh (0.3)\sin (x)\sin (t)$, where $u^{\alpha,\beta}$ denotes here the solution associated with the fractional orders $\alpha$ and $\beta$. Figure \ref{fig:Fig4} shows the evolution of the solution when $\alpha = \beta = 0.8, 0.9, 0.99$, and $1$. Observe how the surface profile of the approximate solution, denoted here by $I_{N_1,N_2}^{\alpha,\beta}\tilde u(x,t)$, converges to $u^{1,1}(x,t)$ as $(\alpha, \beta) \to (1,1)$.

\begin{figure}[H]
\centering
\includegraphics[scale=0.4]{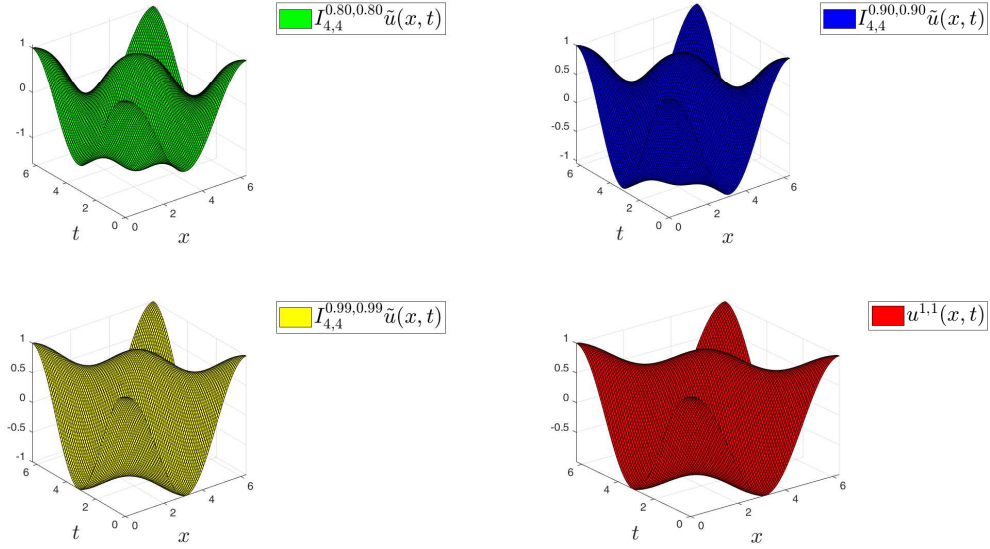}
\caption{The evolution of the approximate solution for some increasing fractional orders. The plots were generated using $100$ equally-spaced nodes in the $x$- and $t$-directions from $0$ to $2 \pi$.}
\label{fig:Fig4}
\end{figure}

\section{Conclusion}
\label{sec:Conc}
We proposed a new class of IVPs of time-dependent one-dimensional FPDEs with variable coefficients and periodic solutions. This class of problems can be solved numerically very accurately and efficiently using the proposed FGPS method. In particular, the FGPS method converts the IVP into a well conditioned linear system of equations using a PS method based on Fourier collocations and Gegenbauer quadratures. The reduced linear system has a sparse block global coefficient matrix that can be generated efficiently using the smart index matrix mapping $\C{N}$ given by Eq. \eqref{eq:smartindm1}, which allows most parts of the global matrix generation process to be optimized and arranged to work on chunks of vectors; thus, the efficiency increases by allowing vectorized operations.
This strategy enables us to solve the reduced linear system of equations very rapidly to nearly within the machine precision using standard linear system solvers. The method converges exponentially for sufficiently smooth periodic solutions using very small collocation mesh grids as proven by Theorem \ref{thm:Jan212022} and verified through extensive numerical simulations. The rigorous analysis of the error conveyed that the truncation errors are due to two sources of errors, namely, the FDC and FII errors. The former is due to the FGPS approximations to the FDs of the FPDE and the latter is induced by the Fourier interpolations of the solution at the collocation mesh grid. We discovered that reducing the FII error by increasing the collocation mesh size generally amplifies the FDC error, so it is recommended to apply the FGPS method using a low size of collocation mesh grid for sufficiently smooth periodic solutions. Corollary \ref{cor:1} shows that the errors are dominated only by the FDC errors when the periodic solution is $\beta$-analytic. We highly anticipate that the idea and results presented in this paper will become fruitful in the future to deal with more general problems involving FPDEs with periodic solutions.

\section*{Declarations}
\subsection*{Competing Interests}
The author declares there is no conflict of interests.

\subsection*{Availability of Supporting Data}
The author declares that the data supporting the findings of this study are available within the article.

\subsection*{Ethical Approval and Consent to Participate and Publish}
Not Applicable.

\subsection*{Human and Animal Ethics}
Not Applicable.

\subsection*{Consent for Publication}
Not Applicable.

\subsection*{Funding}
The author received no financial support for the research, authorship, and/or publication of this article.

\subsection*{Authors' Contributions}
The author confirms sole responsibility for the following: study conception and design, data collection, analysis and interpretation of results, and manuscript preparation.

\bibliographystyle{model1-num-names}
\bibliography{Bib}
\end{document}